\documentclass[reqno]{amsart} 
\usepackage{amsfonts, amsmath, amsthm, amssymb, latexsym, graphicx}

\theoremstyle{plain}
\newtheorem{theorem}{Theorem}[section]
\newtheorem{corollary}[theorem]{Corollary}
\newtheorem{lemma}[theorem]{Lemma}
\newtheorem{proposition}[theorem]{Proposition}
\theoremstyle{definition}

\newtheorem{remark}[theorem]{Remark}

\newtheorem{example}[theorem]{Example}
\newtheorem{remarks-general-heckman-opdam}[theorem]{General Heckman-Opdam theory}
\newtheorem{remarks-general-compact-bc}[theorem]{The compact BC-case}
\numberwithin{equation}{section}

\title[Martingales associated with multivariate Jacobi processes]{Some martingales associated with
  multivariate Jacobi processes and Aomoto's Selberg integral}

\author{ Michael Voit}
\address{Fakult\"at Mathematik, Technische Universit\"at Dortmund,
          Vogelpothsweg 87,
          D-44221 Dortmund, Germany}
\email{michael.voit@math.tu-dortmund.de}

\begin{document}
\date{\today}

\begin{abstract} We study  $\beta$-Jacobi diffusion processes on  alcoves in
$\mathbb R^N$, depending on 3 parameters. 
Using elementary symmetric functions, we present  space-time-harmonic functions and 
martingales for these processes $(X_t)_{t\ge0}$ which are independent from one parameter.
This leads to a formula for $\mathbb E(\prod_{i=1}^N (y-X_{t,i}))$  in terms of classical 
Jacobi polynomials. For $t\to\infty$ this yields a corresponding formula  for Jacobi ensembles and thus
Aomoto's Selberg integral.
\end{abstract}

\maketitle

\section{Introduction}

For an integer $N$ and a constant $\beta\in [0,\infty[$ consider the $\beta$-Jacobi
 (or  $\beta$-MANOVA) ensembles which are $[0,1]^N $-valued random variables $X$ with Lebesgue densities
\begin{equation}\label{start-stationary}
f_{\beta,a_1,a_2}(x) := c_{\beta,a_1,a_2}
 \prod_{i=1}^N x_i^{a_1}(1-x_i)^{a_2} \cdot
 \prod_{1\le i<j\le N}|x_i-x_j|^{\beta}
\end{equation}
with  parameters $a_1,a_2> -1 $ and known normalizations $c_{\beta,a_1,a_2}>0$ which can be expressed 
via Selberg integrals; see the survey \cite{FW}. 
For
$\beta=1,2,4$, and suitable $a_1,a_2$, the variables $X$ appear as spectrum of the classical Jacobi ensembles;
see e.g.~\cite{F}. Moreover, for general $\beta,a_1,a_2$, the variables appear as  eigenvalues of 
the tridiagonal models in \cite{KN}, \cite{K}. Furthermore, for the ordered models on the alcoves
$A_0:=\{x\in\mathbb R^N: \> 0\le x_1\le \ldots\le x_N\le1\}$, the  probability measures with 
densities $N!\cdot f_{\beta,a_1,a_2}$ appear in log gas models as stationary distributions of  diffusions
$(X_t)_{t\ge0}$ with $N$ particles in $[0,1]$; see  \cite{F, Dem, RR1}.
These diffusions and their  stationary distributions  are closely related to Heckman-Opdam
 hypergeometric functions of type  $BC_N$. In particular, the generator  of the 
transition semigroup of  $(X_t)_{t\ge0}$
 is the symmetric part of a Dunkl-Cherednik Laplace operator. Moreover,
 Heckman-Opdam Jacobi polynomials form 
a basis of eigenfunctions, where these polynomials are
 orthogonal w.r.t.~the density  $N!\cdot f_{\beta,a_1,a_2}$ on $A_0$. For the  background see the monograph  \cite{HS} and 
\cite{Dem, RR1, L, BO}.
We point out that  for $\beta=1,2,4$ and suitable $a_1,a_2$, the  diffusion $(X_t)_{t\ge0}$
 and their stationary distributions on $A_0$ are projections of Brownian motions and 
uniform distributions on compact Grassmann manifolds over $\mathbb F=\mathbb R, \mathbb C, \mathbb H$; see
\cite{HS}, \cite{RR2}.

In this paper we use the elementary symmetric polynomials $e_0,e_1,\ldots, e_N$
in $N$ variables and construct poynomials $p_n$ of order $n=1,\ldots,N$ via linear combinations
 such that for 
 suitable exponents $r_n\ge0$, the processes $(e^{r_nt}p_n(X_t))_{t\ge0}$ are martingales where,
 after some  parameter transform,  $p_n$, $r_n$  depend 
only on 2 parameters and not on the third one; see Proposition 
\ref{elementary-symm-martingale} for details.
We  use this result to show that for particular  starting points,
\begin{equation}\label{start-char-pol}
\mathbb E\bigl(\prod_{i=1}^N (y- X_{t,i})\bigr) =  \tilde P_N^{(\alpha,\beta)}(y)\quad\quad\text{for all}\quad t\ge0
\end{equation}
where  $\tilde P_N^{(\alpha,\beta)}$ is a monic Jacobi polynomial on $[0,1]$ where 
 $\alpha,\beta$  depend on the 2 relevant parameters of the  martingale result  
(here, $\beta$ 
is not the $\beta$ in (\ref{start-stationary})). In the limit $t\to\infty$,  (\ref{start-char-pol})
 leads to a corresponding formula for the expectation for the random variable $X$ with density 
 in (\ref{start-stationary})  and to Aomoto's Selberg integral \cite{A}.
 Corresponding results for classical Hermite and Laguerre ensembles are given in \cite{DG, FG};
 for these results for related multivariate Bessel processes we refer to \cite{KVW}.
Clearly, (\ref{start-char-pol}) admits an interpretation for
 characteristic polynomials of classical Jacobi ensembles
and the tridiagonal $\beta$-Jacobi models in \cite{KN}.
 
The proof of  the  martingale result  
 relies on the stochastic differential equation for the diffusion  $(X_t)_{t\ge0}$;
 for the general background here we refer to \cite{P, RW}.

This paper is organized as follows. In Section 3 we briefly recapitulate  some facts on
 Heckman-Opdam Jacobi polynomials of type $BC$, the  associated 
Dunkl-Cherednik Laplace operator, and the transition semigroup of the diffusion
 $(X_t)_{t\ge0}$.
In Section 3 we then use stochastic analysis to derive our martingales.
 Section 4 is then devoted to  (\ref{start-char-pol}).
 In Section 5 we then discuss some connection between the martingale result and  Heckman-Opdam
Jacobi polynomials.

A  comment about  notations and normalizations:
 We  start in Section 2 with a brief survey on the compact  Heckman-Opdam theory
of type $BC$ with multiplicity parameters $k_1,k_2,k_3\ge0$. We 
 transfer all results from the trigonometric case to
the interval $[-1,1]$, and start with transformed parameters $\kappa,p,q$
in the SDE approach in Section 3.  We there follow
\cite{Dem} where we replace his  $\beta$  by $\kappa\ge0$, in order to avoid any confusion with the
 classical parameters $\alpha,\beta>-1$ of the  one-dimensional Jacobi polynomials $P_N^{(\alpha,\beta)}$.
The choice of the interval  $[-1,1]$ instead of $[0,1]$ as in (\ref{start-stationary}) or \cite{Dem}
is  caused by the fact that the results should fit to the  $P_N^{(\alpha,\beta)}$.

\section{ Heckman-Opdam Jacobi polynomials of type $BC$}

We first recapitulate some general facts on  Heckman-Opdam theory from  \cite{HS}.

\begin{remarks-general-heckman-opdam}
 Let  $(\frak a, \langle\,.\,,\,.\,\rangle) $ be a Euclidean space of dimension $N$. 
 Let $R$ be a crystallographic, possibly not reduced root system in $\frak a$
 with associated reflection group $W$. 
 Fix a positive subsystem $R_+$ of $R$ and a $W$-invariant multiplicity function $k:R\to[0,\infty[$. 
 The Cherednik operators associated with $R_+$ and $k$ are 
 \begin{equation}\label{def-dunkl-cher-op} 
D_\xi(k)f(x) 
= \partial_\xi f(x) + \sum_{\alpha \in R_+}  \frac{ k(\alpha) \langle \alpha, \xi\rangle }{1-e^{-\langle \alpha,x\rangle}}
 (f(x)-f(\sigma_\alpha(x)) -\langle \rho(k), \xi\rangle f(x) 
\end{equation} 
for $\xi \in \mathbb R^n$ with  the half-sum
 $ \rho(k) := \frac{1}{2} \sum_{\alpha \in R_+} k(\alpha) \alpha$ of positive roots.
 
  The $D_\xi(k)$ ($\xi \in \frak a$) commute, and for each $\lambda \in \frak a _\mathbb C$ there exists a unique analytic function $G(\lambda,k; .\,)$ on a common $W$-invariant tubular neighborhood of $\frak a$ in the complexification 
 $\frak a_{\mathbb C}$, the so called Opdam-Cherednik kernel,  satisfying
 \begin{equation}\label{Opdam-Cherednik}
  D_\xi(k)G(\lambda,k;\,.\,) = \langle \lambda, \xi \rangle\, G(\lambda,k; \,.\,) \>\> \forall \, \xi \in \frak a ; \,\,G(\lambda, k; 0) = 1.
\end{equation} 
The hypergeometric function associated with $R$ is  defined by
$$ F(\lambda, k;z) = \frac{1}{|W|} \sum_{w\in W} G(\lambda, k; w^{-1}z).$$
For the Heckman-Opdam polynomials, we write
$\alpha^\vee = \frac{2\alpha}{\langle \alpha, \alpha \rangle}$ for $\alpha \in R$ and
 use the weight lattice and the set of dominant weights associated with $R$ and $R_+$, 
$$ P = \{\lambda \in \frak a: \langle \lambda, \alpha ^\vee \rangle \in \mathbb Z \>\> \forall \alpha \in R\,\}, \quad
P_+ = \{\lambda \in P: \langle \lambda, \alpha^\vee\rangle \geq 0\,\, \forall \alpha \in R_+\,\}\supset R_+ $$ 
 where
 $P_+$ carries the usual dominance order. 
Let
$ \mathcal T:= \text{span}_{\mathbb C}\{e^{i\lambda}, \, \lambda \in P\}$
the vector space of trigonometric polynomials associated with $R$.
The orbit sums 
$$ M_\lambda = \sum_{\mu \in W\!\lambda} e^{i\mu}\,, \quad \lambda \in P_+$$
form a basis of the subspace $\mathcal T^W$ of $W$-invariant polynomials in $\mathcal T$.
For $ Q^\vee := \text{span}_{\mathbb Z}\{\alpha^\vee, \, \alpha \in R\}$,
 consider the  torus $\mathbb T = \frak a/2\pi Q^\vee$ with 
the weight function 
\begin{align}\label{weight-trig} \delta_k(t) := 
\prod_{\alpha \in R_+} \Bigl|\sin\Bigr(\frac{\langle \alpha, t\rangle}{2}\Bigr)\Bigr|^{2k_\alpha}.
\end{align}
The Heckman-Opdam polynomials associated with $R_+$ and $k$ are given by
$$ P_\lambda(k;z) := M_\lambda(z) + \sum_{\nu < \lambda} c_{\lambda\nu}(k) M_\nu(z) \quad (\lambda \in P_+\,, z\in \frak a_{\mathbb C})$$
where the  $c_{\lambda\nu}(k)\in \mathbb R$ are  determined  by the condition that $P_\lambda(k;\,.\,)$
 is orthogonal to $M_\nu$ in $L^2(\mathbb T, \delta_k)$ for $\nu \in P_+$ with $\nu<\lambda$.
 It is known that 
 $\{P_\lambda(k, \,.\,), \lambda\in P_+\,\}$ is an orthonormal basis of the space $L^2(\mathbb T, \delta_k)^W$
 of all $W$-invariant functions in $L^2(\mathbb T, \delta_k)$. 
By \cite{HS}, the normalized polynomials
$$ R_\lambda(k,z):= P_\lambda(k;z)/P_\lambda(k;0)$$
satisfy
\begin{equation}\label{relation-pol-hyper}
 R_\lambda( k,z) = F(\lambda+\rho(k),k;iz).
\end{equation}
We next introduce the Heckman-Opdam Laplacian
$$\Delta_k:= \sum_{j=1}^N D_{\xi_j}(k)^2 \>\> - \>\> \|\rho(k)\|_2^2$$
with  an orthonormal basis $\xi_1,\ldots,\xi_N$ of $\frak a$.
The operator $\Delta_k$ does not depend on the basis and, by \cite{Sch1, Sch2},
 has for a $W$-invariant function $f$ the form
$$\Delta_k f(x)= \Delta f(x)+\sum_{\alpha \in R_+} k(\alpha) \coth\Bigl( \frac{\langle \alpha, x\rangle}{2}\Bigr)\cdot
 \partial_\alpha f(x).$$
If we take the factor $i$ in (\ref{relation-pol-hyper}) into account as in \cite{RR1}, we now consider the operator
\begin{equation}\label{generator-general}
\tilde\Delta_kf(t):=\Delta f(t)+\sum_{\alpha \in R_+} k(\alpha) \cot\Bigl( \frac{\langle \alpha, t\rangle}{2}\Bigr)\cdot
 \partial_\alpha f(t).
\end{equation}
By construction, the $P_\lambda$ ($\lambda\in P_+$) 
are eigenfunctions of $\tilde\Delta_k$ with
with  eigenvalues $-\langle \lambda,\lambda+2\rho(k)\rangle\le0$.
This is used in \cite{RR1} to construct the transition densities of the diffusions with the  generators $L_k$.
\end{remarks-general-heckman-opdam}

\begin{remarks-general-compact-bc}
We now turn to the nonreduced root system
 $$ R= BC_N = \{\pm e_i, \pm 2 e_i,  \pm(e_i \pm e_j); \>\> 1\leq i < j \leq N\} \subset \mathbb R^N $$
 with  weight lattice  $P = \mathbb Z^n$ and  torus $\mathbb T= (\mathbb R/2\pi \mathbb Z)^N.$
 The mulitplicities on $R$ are written as  $k=(k_1,k_2, k_3)$ with $k_1, k_2, k_3$ as the values on the roots $e_i, 2e_i, e_i \pm e_j$.  
 Now fix a positive subsystem $R_+$ and consider
the associated normalized Heckman-Opdam Jacobi polynomials $R_\lambda= R_\lambda^{BC}$ ($\lambda\in \mathbb Z_+^N$)
as  e.g.~in \cite{BO, L, RR1}.
 (\ref{weight-trig}) and   some calculation  show that the  polynomials $\widetilde R_\lambda$ with
 $$ \widetilde R_\lambda(\cos t):= R_\lambda(k;t) \quad\quad(\lambda \in \mathbb Z_+^N)$$
  form an  orthogonal basis of
$L^2(A_N, w_k)$ on the alcove
$$A_N:=\{x\in\mathbb R^N| \> -1\leq x_1\leq ...\leq x_N\leq 1\}$$
with the weight function 
\begin{equation}\label{weight-general} w_k(x)
 := \prod_{i=1}^N (1-x_i)^{k_1+k_2-1/2}(1+x_i)^{k_2-1/2} \cdot \prod_{i<j}|x_i-x_j|^{2k_3}.
\end{equation}
The  operator $\tilde\Delta_k$ from (\ref{generator-general}) is given by
\begin{align}\label{generator-trigonometric}
\tilde\Delta_k f(t):=\Delta f(t)+\sum_{i=1}^N &\Biggl(k_1 \cot(\frac{t_i}{2})+2k_2 \cot(t_i)\\
&+k_3\sum_{j: j\ne i}\Bigl(\cot(\frac{t_i-t_j}{2})+\cot(\frac{t_i+t_j}{2})\Bigr)\Biggr)\partial_{i}f(t).
\notag\end{align}
The substitutions $x_i=\cos t_i$  and 
elementary calculations lead to the  operator  
\begin{equation}\label{generator-algebraic}
 L_kf(x):= \sum_{i=1}^N (1-x_i^2)f_{x_i,x_i}(x)+ \sum_{i=1}^N\Biggl( -k_1-(1+k_1+2k_2)x_i
+2k_3\sum_{j: j\ne i} \frac{1-x_i^2}{x_i-x_j}\Biggr)f_{x_i}(x).
\end{equation}
In summary, the Heckman-Opdam Jacobi 
polynomials $ \widetilde R_\lambda$ are eigenfunctions
of $ L_k$ with  eigenvalues $-\langle \lambda,\lambda+2\rho(k)\rangle$ where 
$\rho(k)$ has the  coordinates 
\begin{equation}\label{rho-component}
\rho(k)_i=\bigl(k_1+2k_2+2k_3(N-i)\bigr)/2 \quad\quad (i=1,\ldots,N).
\end{equation}
By \cite{RR1}, $ L_k$ 
is the generator of a Feller semigroup with transition operators whose smooth densities admit
series expansions involving the $ \widetilde R_\lambda$. Moreover, by standard stochastic calculus, the associated Feller 
processes $(X_t)_{t\ge0}$ with coordinates $X_{t,i}$ should be solutions of the SDEs
\begin{equation}\label{SDE-alcove-k}
dX_{t,i} = \sqrt{2(1-X_{t,i}^2)}\> dB_{t,i} + \Bigl( -k_1-(1+k_1+2k_2)X_{t,i} +2k_3\sum_{j: j\ne i} \frac{1-X_{t,i}^2}{X_{t,i}-X_{t,j}}\Bigr)dt.
\end{equation}
for $i=1,\ldots,N$ and an $N$-dimensional Brownian motion $(B_t)_{t\ge0}$. 
In fact, it is shown in Theorem 2.1 of \cite{Dem} that for any starting point $x$ in the interior of $A_N$ and all
$k_1,k_2,k_3>0$, the SDE (\ref{SDE-alcove-k}) has a unique strong solution $(X_t)_{t\ge0}$ 
where the paths are reflected when they meet the boundary $\partial A_N$ of $A_N$.
 We  study these Jacobi processes in the next section.
\end{remarks-general-compact-bc}

\begin{example}
For $N=1$, the   $ \widetilde R_\lambda$ are
 one-dimensional Jacobi polynomials
\begin{align}\label{jacobi-pol-def}
P_n^{(\alpha,\beta)}(x)&:= \binom{n+\alpha}{n}  \>_2F_1\bigl(-n, n+\alpha + \beta+1; \alpha +1;(1-x)/2\bigr) \\
&= \sum_{k=0}^{n} \binom{n}{k} 
\frac{(n+\alpha+\beta+1)_k(\alpha+k+1)_{n-k}}{n!} \Bigl(\frac{x-1}{2}\Bigr)^k\notag
\end{align}
 for  $\alpha,\beta>-1$, where the $P_n^{(\alpha,\beta)}$  are orthogonal 
 w.r.t.~the weights $(1-x)^\alpha(1+x)^\beta$   on $]-1,1[$; see Ch.~4 of  \cite{S}.
With these notations we see from (\ref{weight-general}) that
 $$\binom{n+\alpha}{n} \tilde R_n^{BC_1}(k;.) =  P_n^{(\alpha, \beta)} \quad  \text{ with } \quad
\alpha = k_1+k_2 -\frac{1}{2}, \, \beta = k_2-\frac{1}{2}.$$
Moreover, (\ref{generator-algebraic}) corresponds with the classical differential equation for the $P_n^{(\alpha,\beta)}$.
\end{example}

\section{Some martingales related to Jacobi processes}

In this section we study the $\beta$-Jacobi processes $(X_t)_{t\ge0}$ on $A_N$ which satisfy (\ref{SDE-alcove-k}). 
We  follow \cite{Dem} and introduce new parameters $p,q,\kappa>0$ instead $k_1,k_2,k_3$ where we replace the parameter
$\beta$ in  \cite{Dem} by $\kappa$ in order to avoid problems with the  classical Jacobi polynomials below.
For  $\kappa>0$ and 
$p,q>N-1+1/\kappa$, we now
 define  the  Jacobi process $(X_t)_{t\ge0}$  as the unique strong  solution 
of the SDEs
\begin{align}\label{SDE-alcove}
dX_{t,i} &= \sqrt{2(1-X_{t,i}^2)}\> dB_{t,i} +\kappa\Bigl((p-q) -(p+q)X_{t,i}\\
& \quad\quad\quad\quad\quad\quad \quad\quad\quad\quad\quad +
\sum_{j: \> j\ne i}\frac{(1+X_{t,i})(1-X_{t,j})+(1+X_{t,j})(1-X_{t,i})}{X_{t,i}-X_{t,j}}\Bigr)dt \notag\\
 &= \sqrt{2(1-X_{t,i}^2)}\> dB_{t,i} +\kappa\Bigl((p-q) -(p+q)X_{t,i} +
2\sum_{j: \> j\ne i}\frac{1-X_{t,i}X_{t,j}}{X_{t,i}-X_{t,j}}\Bigr)dt
\notag\\
 &=\sqrt{2(1-X_{t,i}^2)}\> dB_{t,i} +\kappa\Bigl((p-q) +(2(N-1)-(p+q))X_{t,i}\notag\\
& \quad\quad\quad\quad\quad\quad \quad\quad\quad\quad\quad +
2\sum_{j: \>j\ne i}\frac{1-X_{t,i}^2}{X_{t,i}-X_{t,j}}\Bigr)dt
\notag
\end{align}
for $i=1,\ldots,N$
with an $N$-dimensional Brownian motion $(B_t)_{t\ge0}$
where the paths of  $(X_t)_{t\ge0}$ are reflected on $\partial A_N$
and where we start in some point in the interior of $A_N$; see 
Theorem 2.1 of \cite{Dem}.
Clearly the SDEs (\ref{SDE-alcove}) and (\ref{SDE-alcove-k}) are equal for
\begin{equation}\label{parameter-change-k-p}
\kappa=k_3, \quad q= N-1+\frac{1+2k_1+2k_2}{2k_3}, \quad p= N-1+\frac{1+2k_2}{2k_3}.
\end{equation}
 It is  known
 by \cite{Dem, Dou} that  for
 $\kappa\ge 1$ and 
$p,q\ge N-1+2/\kappa$,
the process does not meet  $\partial C_N^A$  almost surely.

Besides the original processes for $\kappa>0$ we also consider the transformed processes 
$(\tilde X_{t}:=X_{t/\kappa})_{t\ge0}$.
We use the obvious formulas
$$ \int_0^t Z_{s/\kappa} \> ds=\kappa \int_0^{t/\kappa } Z_{s} \> ds \quad\text{and}\quad
 \int_0^t Z_{s/\kappa} \> d\tilde B_{s}=\sqrt\kappa \int_0^{t/\kappa } Z_{s} \>  dB_s$$
with  Brownian motions  $(B_t)_{t\ge0}, \>(\tilde B_t)_{t\ge0} $ starting in $0$ related by
 $\tilde B_t= \sqrt\kappa \cdot B_{t/\kappa}$. We then obtain the renormalized  SDEs
\begin{equation}\label{SDE-alcove-normalized}
d\tilde X_{t,i} = 
\frac{\sqrt 2}{\sqrt\kappa } \sqrt{1-\tilde X_{t,i}^2}\> d\tilde B_{t,i} +\Bigl((p-q) -(p+q)\tilde X_{t,i} +
2\sum_{j\ne i}\frac{1-\tilde X_{t,i}\tilde X_{t,j}}{\tilde X_{t,i}-\tilde X_{t,j}}\Bigr)dt
\end{equation}
for $i=1,\ldots,N$. The generator of the diffusion semigroup associated with $(\tilde X_{t})_{t\ge0}$
is the operator $\tilde L_k:=\frac{1}{\kappa} L_k$.

 We now derive some results for symmetric polynomials of 
$(\tilde X_{t})_{t\ge0}$ and $(\tilde X_{t})_{t\ge0}$. 
For this we recapitulate that the elementary
symmetric polynomials $e_n^{m}$  in $m$ variables for $n=0,\ldots,m$  are characterized by
\begin{equation}\label{symmetric-poly}
\prod_{j=1}^m (z-x_j) = \sum_{j=0}^{m}(-1)^{m-j}  e_{m-j}^{m}(x) z^j \quad\quad (z\in\mathbb C, \> x=(x_1,\ldots,x_m)).
\end{equation}
In particular,
$e_0^{m}=1, \> e_1^{m}(x)=\sum_{j=1}^m x_j ,\ldots, e_m^{m}(x)=\prod_{j=1}^m x_j$. 

We need a further notation: For  a non-empty set
 $S\subset \{1,\ldots,N\}$, let $\tilde X_{t}^S$ be the $\mathbb R^{|S|}$-valued
 variable with the 
coordinates $\tilde  X_{t,i}$ for $i\in S$ in the natural ordering on $S\subset \{1,\ldots,N\}$. 
 We need the following technical observation:

\begin{lemma}\label{symmetric-pol-in-t}
 For all $r\in\mathbb R$, $n=0,1,\ldots,N$, $\kappa\ge 1$ and $p,q$ with
$p,q\ge N-1+2/\kappa$,
\begin{align}d(e^{rt}&\cdot e_n^N(\tilde X_{t}))=
\frac{\sqrt 2 \cdot e^{rt}}{\sqrt\kappa}\sum_{j=1}^N\sqrt{1-\tilde X_{t,j}^2}\cdot 
e_{n-1}^{N-1}(\tilde X_{t}^{\{1,\ldots,N\}\setminus\{j\}}) d\tilde B_{t,j} 
\notag\\
&+e^{rt}\Biggl((r-n(p+q-n+1)) e_n^N(\tilde X_{t})
+ (p-q)(N-n+1)e_{n-1}^N(\tilde X_t) 
\notag\\
&\quad\quad\quad\quad -(N-n+2)(N-n+1)
e_{n-2}^N(\tilde X_{t})\Biggr)dt
\notag\end{align}
where, for $n=0,1$, we assume that $e_{-2}\equiv e_{-1}\equiv 0$.
\end{lemma}

\begin{proof} We are in the situation where the process does not meet the boundary.
 Ito's formula and the SDE (\ref{SDE-alcove-normalized}) show that
$$d(e^{rt}\cdot e_n^N(\tilde X_{t}))= 
r\cdot e^{rt}\cdot e_n^N(\tilde X_{t}) \> dt+
 e^{rt}\sum_{j=1}^N  e_{n-1}^{N-1}(\tilde X_{t}^{\{1,\ldots,N\}\setminus\{j\}})\> d\tilde X_{t,j}.$$
Therefore, by the second line of the  SDE  (\ref{SDE-alcove-normalized}), and with
$$dM_t:= \frac{\sqrt 2\cdot e^{rt}}{\sqrt\kappa}\sum_{j=1}^N\sqrt{1-\tilde X_{t,j}^2}\cdot
 e_{n-1}^{N-1}(\tilde X_{t}^{\{1,\ldots,N\}\setminus\{j\}})\> d\tilde B_{t,j},$$ 
\begin{align}\label{elementary-symm-1}
d(&e^{rt}\cdot e_n^N(\tilde X_{t}))= 
r\cdot e^{rt}  e_n^N(\tilde X_{t}) \> dt+ dM_t\\
&\quad+ e^{rt}\Biggl((p-q)\sum_{j=1}^N e_{n-1}^{N-1}(\tilde X_{t}^{\{1,\ldots,N\}\setminus\{j\}})
-(p+q)\sum_{j=1}^N  e_{n-1}^{N-1}(\tilde X_{t}^{\{1,\ldots,N\}\setminus\{j\}})\cdot \tilde X_{t,j}
\notag\\
&\quad\quad\quad\quad\quad+ 2\sum_{i,j;\>  i\ne j}\frac{1-\tilde X_{t,j}\tilde X_{t,i}}{\tilde X_{t,j}-\tilde X_{t,i}}
e_{n-1}^{N-1}(\tilde X_{t}^{\{1,\ldots,N\}\setminus\{j\}})
\Biggr)dt.\notag
\end{align}
Simple combinatorial computations for $i\neq j$ (cf. (2.10), (2.11) in \cite{VW}) yield
\begin{align}\label{elementary-symm-2}
&\sum_{j=1}^N e_{n-1}^{N-1}(\tilde X_{t}^{\{1,\ldots,N\}\setminus\{j\}})=(N-n+1)\cdot  e_{n-1}^{N}(\tilde X_{t}),\\
&\sum_{j=1}^N e_{n-1}^{N-1}(\tilde X_{t}^{\{1,\ldots,N\}\setminus\{j\}})\tilde X_{t,j}=n\cdot  e_{n}^{N}(\tilde X_{t})\notag
\end{align}
as well as
\begin{equation}\label{elementary-symm-2a}
  e_{n-1}^{N-1}(\tilde X_{t}^{\{1,\ldots,N\}\setminus\{j\}})- e_{n-1}^{N-1}(\tilde X_{t}^{\{1,\ldots,N\}\setminus\{i\}})=
 (\tilde X_{t,i}-\tilde X_{t,j})e_{n-2}^{N-2}(\tilde X_{t}^{\{1,\ldots,N\}\setminus\{i,j\}})
\end{equation}
and 
\begin{equation}\label{elementary-symm-3}
 \sum_{i,j=1,\ldots,N; i\ne j}e_{n-2}^{N-2}(\tilde X_{t}^{\{1,\ldots,N\}\setminus\{i,j\}})= 
(N-n+2)(N-n+1)e_{n-2}^N(\tilde X_{t}).
\end{equation}
 (\ref{elementary-symm-1})-(\ref{elementary-symm-3}) now imply
\begin{align}
d(e^{rt}&\cdot e_n^N(\tilde X_{t}))=  dM_t + e^{rt}\Biggl( r\cdot e_n^N(\tilde X_{t}) + 
(p-q)(N-n+1)  e_{n-1}^N(\tilde X_{t})\notag\\
& - n(p+q) e_n^N(\tilde X_{t})
- (N-n+2)(N-n+1)e_{n-2}^N(\tilde X_{t})+ n(n-1) e_n^N(\tilde X_{t})\Biggr)dt.
\notag
\end{align}
This leads to the lemma for $n\ge2$. 
An inspection of the proof shows that all formulas are also valid for $n=0,1$ with the convention of the lemma.
\end{proof}

 Lemma \ref{symmetric-pol-in-t}  leads to the following martingales w.r.t.~the canonical
 filtration of the Brownian motion $(\tilde B_t)_{t\ge0}$.

\begin{proposition}\label{elementary-symm-martingale}
Let  $n\in\{1,\ldots,N\}$, $\kappa>0$ and
$p,q>N-1+1/\kappa$. Put
$$r_n:=n(p+q-n+1).$$
Then there exist coefficients $c_{n,l}\in\mathbb R$ for $l=0,\ldots, n-1$
such that
for all starting points $x_0$ in the interior of $C_N^A$ and the Jacobi process $(\tilde X_{t})_{t\ge0}$
with parameters $\kappa,p,q$,
 the process
\begin{equation}\label{mart-formel}
\Biggl( e^{r_nt}\cdot
\Biggl( e_n^N(\tilde X_{t}) + \sum_{l=1}^{n} c_{n,l} \cdot e_{n-l}^N(\tilde X_{t})\Biggr)\Biggr)_{t\ge0}
\end{equation}
is a martingale. The $c_{n,l}$ and $r_n$ do not
depend  on $\kappa$.
\end{proposition}

\begin{proof} 
We first assume that $\kappa\ge 1$ and $p,q\ge N-1+2/\kappa$ as in Lemma \ref{symmetric-pol-in-t}.
  Lemma \ref{symmetric-pol-in-t} here yields that for $l=0,\ldots, n$, 
the processes
$ (e^{r_nt}\cdot e_{n-l}^N(\tilde X_{t}))_{t\ge0}$ are linear combinations of the processes 
$I_{n-l-j}:=(\int_0^t e^{r_ns}\cdot e_{n-l-j}^N(\tilde X_{s})\> ds)_{t\ge0}$ for $j=0,1,2$
 up to the addition of some integrals w.r.t.~$(\tilde B_t)_{t\ge0}$. As the 
integrands of these  Brownian integrals are bounded, these Brownian integrals are obviously
martingales.

Let us now consider all  linear combinations of the processes 
$I_{n-l-j}$ for $l\ge0$ and  $j=0,1,2$.
The definition of $r_n$ and Lemma \ref{symmetric-pol-in-t} ensure that the summand $I_n$ does not appear.
Moreover, if we put
\begin{equation}\label{recurrence-1a}
c_{n,1}:= \frac{(p-q)(N-n+1)}{r_n-r_{n-1}}= \frac{(p-q)(N-n+1)}{p+q-2n+2},
\end{equation}
we see that the summand  $I_{n-1}$ also does not appear. If we now define
\begin{align}\label{recurrence-1b}
c_{n,l}:=& \frac{(p-q)(N-n+l)c_{n,l-1}-(N-n+l)(N-n+l+1)c_{n,l-2}}{r_n-r_{n-l}}\\
=& \frac{(p-q)(N-n+l)c_{n,l-1}-(N-n+l)(N-n+l+1)c_{n,l-2}}{l(p+q-2n+l+1)}
\notag
\end{align}
for $l=2,\ldots,n$ with $c_{n,0}=1$ in an recursive way, we obtain from Lemma \ref{symmetric-pol-in-t} 
 that the summands  $I_{n-l}$ also do not appear.
In summary, we conclude that the proposition holds for $\kappa\ge 1$ and $p > q\ge N-1+2/\kappa$.

We now use Dynkin's formula (see e.g. Section III.10 of \cite{RW}) which implies that
 the symmetric functions
 $$ f_{N,n}:A_N\times [0,\infty[\to \mathbb R, \quad (x,t)\mapsto
 e^{r_nt}\cdot
\Bigl( e_n^{N}(x) + \sum_{l=1}^{n} c_{n,l} \cdot e_{n-l}^{N}(x)\Bigr) $$
are space-time-harmonic w.r.t. the generators of the renormalized 
Jacobi processes $(\tilde X_t)_{t\ge0}$   for $\kappa\ge 1$ and $p > q\ge N-1+2/\kappa$,
i.e., 
\begin{equation}\label{diff-op-space-time}
(\frac{\partial}{\partial t}+ \tilde L_k)f_{N,n}\equiv 0
\end{equation}
 for the parameters related via
(\ref{parameter-change-k-p}). As the left hand side of (\ref{diff-op-space-time}) is analytic in $p,q,\kappa$,
  analytic continuation shows that $f_{N,n}$ is  space-time-harmonic
 also  for all  $\kappa>0$ and 
$p>q>N-1+1/\kappa$ with the corresponding coefficients  $c_{n,l}= c_{n,l}(p,q,N)$
via (\ref{recurrence-1a}) and (\ref{recurrence-1b}). Dynkin's formula  now yields the
 proposition in general.
\end{proof}

The independence of $r_n$ and $c_{n,l}$ from $\kappa$ in  Proposition \ref{elementary-symm-martingale} is not surprising
by the  space-time-harmonicity argument. In fact, $\kappa$ appears in the differential operator
  in (\ref{diff-op-space-time}) only as a constant factor in the classical Laplace operator $\Delta$.
As $\Delta e_{j}^{N}\equiv 0$ for all $j$, the independence of  $\kappa$
is obvious.

\begin{remark}
The recurrence formulas  (\ref{recurrence-1a}) and (\ref{recurrence-1b}) for the $c_{n,l}$
 can be simplified slightly; we however do not have 
 a closed formula for  $\tilde c_{n,l}$ except for  $p=q$.

In the case  $p=q$ we have $ c_{n,l}=0$ for $l$ odd,  and
\begin{equation}\label{c-k-2l}
 c_{n,2l}=\frac{(-1)^l (N-n+2)_{2l}}{ l!\cdot 2^l \cdot (p+q-2n+3)(p+q-2n+5)\cdots(p+q-2n+2l+1)}
\end{equation}
for $l=1,\ldots,\lfloor n/2\rfloor$. This follows easily from (\ref{recurrence-1a}) and  (\ref{recurrence-1b}). 
\end{remark}

We  return to Lemma \ref{symmetric-pol-in-t} and  Proposition \ref{elementary-symm-martingale}.
An inspection of the proofs shows that both results are also valid for  $\kappa=\infty$
 in which case the
 SDE (\ref{SDE-alcove-normalized}) is an ODE, and the process $(\tilde X_{t})_{t\ge0}$ is deterministic
 whenever so is the initial condition for $t=0$. There are several limit theorems 
(laws of large numbers, CLTs) for the limit transition $\kappa\to\infty$; see \cite{HV}. 
In particular,  Proposition \ref{elementary-symm-martingale}
for $\kappa\in]0,\infty]$ leads to:

\begin{corollary}\label{constant-expectation-general}
For any  starting point  $x_0$ in the interior of $A_N$,
 let $(\tilde X_{t})_{t\ge0}$ be the 
associated normalized Jacobi process 
with   $\kappa\in]0,\infty]$ and $p,q>N-1+1/\kappa$.
Then there are constants $a_{n,l}\in\mathbb R$ for $0\le l\le n\le N$ such that
$$\mathbb E(e_n^N(\tilde X_{t}) )=\sum_{l=0}^n a_{n,l} e^{-r_lt}$$
with $r_0=0$ where the coefficients $a_{n,l}$ and the exponents $r_l$ 
 do not depend on $\kappa$.
\end{corollary}

\begin{proof} 
 Proposition \ref{elementary-symm-martingale} shows that
$$\mathbb E(e_n^N(\tilde X_{t}) )=e^{-r_nt}\Bigl(e_n^N(x_0)+ \sum_{l=0}^{n-1} c_{n,l}  e_l^N(x_0)\Bigr)-
\sum_{l=0}^{n-1} c_{n,l}  E(e_l^N(\tilde X_{t}) ).$$
As $a_{n,l}:= e_n^N(x_0)+ \sum_{l=0}^{n-1} c_{n,l}  e_l^N(x_0)$ is independent of $\kappa$, the corollary follows by
 induction.
\end{proof}

\section{Results for a special starting point}

We now choose  special starting points. 
For this we use the ordered zeros of special Jacobi polynomials $P_N^{(\alpha,\beta)}$ as introduced in Section 2.
We  need the following characterization of the ordered zeros 
 $z_1\le \ldots\le z_N$ of  $P_N^{(\alpha,\beta)}$ due to Stieltjes, which is presented 
in \cite{S} as Theorem 6.7.1:

\begin{lemma} Let $(x_1,\ldots,x_N)\in A_N$. Then $(x_1,\ldots,x_N)=(z_1,\ldots,z_N)=:z$ if and only if
for all  $j=1,...,N$,
\begin{equation}\label{eq-zeros}
\sum_{i=1,\ldots,N,i\neq j}\frac{1}{x_j-x_i}+\frac{\alpha+1}{2}\frac{1}{x_j-1}+\frac{\beta+1}{2}\frac{1}{x_j+1}=0.
\end{equation}
\end{lemma}

We now return to the normalized Jacobi processes $(\tilde X_t)_{t\ge0}$ with parameters $p,q,\kappa$. 
We write the drift parts in the SDEs (\ref{SDE-alcove-normalized}) as
\begin{align}\label{drift-part}
(p&-q) -(p+q)\tilde X_{t,i} +
2\sum_{j: \>j\ne i}\frac{1-\tilde X_{t,i}\tilde X_{t,j}}{\tilde X_{t,i}-\tilde X_{t,j}}\notag\\
&= (p-q)+ (2(N-1)-(p+q))\tilde X_{t,i}+2(1-\tilde X_{t,i}^2)\sum_{j: \> j\ne i}\frac{1}{\tilde X_{t,i}-\tilde X_{t,j}}
\notag\\
&=2(1-\tilde X_{t,i}^2)\cdot\Bigl( \frac{p-(N-1)}{2} \frac{1}{\tilde X_{t,i}+1}+
 \frac{q-(N-1)}{2} \frac{1}{\tilde X_{t,i}-1}+ \sum_{j: \> j\ne i}\frac{1}{\tilde X_{t,i}-\tilde X_{t,j}}\Bigr)
\end{align}
for $i=1,\ldots,N$. 
We now compare (\ref{drift-part}) with (\ref{eq-zeros}) and obtain:

\begin{corollary}\label{constant-solution}
Let $p,q>N-1$, and put
\begin{equation}\label{alpha-beta}
\alpha:= q-N>-1, \quad \beta:= p-N>-1.
\end{equation}
Then, for $\kappa=\infty$, the SDE (\ref{SDE-alcove-normalized}) is an ODE which has the vector $z\in A_N$
of the preceding lemma as unique constant solution.
\end{corollary}

We now combine this with Corollary \ref{constant-expectation-general} and obtain:

\begin{corollary}\label{constant-solution-2}
Let $\kappa\in]0,\infty]$, $p,q>N-1$, and take the vector $z\in A_N$
of the preceding lemma as starting point. Let   $(\tilde X_t)_{t\ge0}$  be the associated normalized 
 Jacobi process. Then, for all $n=0,1,\ldots,N$ and $t\ge0$, 
$$\mathbb E(e_n^N(\tilde X_t))=e_n^N(z).$$
In particular, this expectation does not depend on $t$ and $\kappa$.
\end{corollary}

\begin{proof} This is clear for $\kappa=\infty$ by Corollary \ref{constant-solution}. As 
$\mathbb E(e_n^N(\tilde X_t))$ is independent of  $\kappa\in]0,\infty]$ by Corollary
 \ref{constant-expectation-general},
the result is clear.
\end{proof}

We immediately obtain:

\begin{corollary}\label{constant-solution-3}
Let $\kappa\in]0,\infty[$, $p,q>N-1$, and  $z\in A_N$ as above. Then for the  associated Jacobi process 
 $( X_t)_{t\ge0}$ starting in $z$, and all  $n=0,1,\ldots,N$ and $t\ge0$,
 $\mathbb E(e_n^N( X_t))=e_n^N(z).$
\end{corollary} 

As an application, we get the following result:

\begin{theorem}\label{det-formula} 
Let $\kappa\in]0,\infty[$, $p,q>N-1$, and $\alpha,\beta$ as well as  $z\in A_N$ as above. Let 
 $( X_t)_{t\ge0}$  be  the  associated Jacobi process starting in $z$. 
Then, for all $t\ge0$,
\begin{equation}\label{det-form-1}
\mathbb E\bigl(\prod_{i=1}^N (y- X_{t,i})\bigr) = \frac{1}{l_N^{(\alpha,\beta)}} \cdot P_N^{(\alpha,\beta)}(y)
  \quad\quad\text{for}\quad y\in \mathbb R
\end{equation}
with the leading coefficient $l_N^{(\alpha,\beta)}$ of $ P_N^{(\alpha,\beta)}$.
Moreover, for $n=0,1,\ldots, N$,
\begin{align}\label{det-form-2}
& \mathbb E\bigl(e_{n}^N(X_{t})\bigr)= e_{n}^N(z)\\
&=\frac{2^{N}}{\binom{2N+\alpha+\beta}{N}}
\sum_{l=N-n}^N  (-1)^{N-l} \binom{N}{l} \binom{l}{N-n} 
\frac{(N+\alpha+\beta+1)_l(\alpha+l+1)_{N-l}}{N!2^l}.\notag
\end{align}
\end{theorem}

\begin{proof} Corollary \ref{constant-solution-3} shows that
\begin{align}\label{det-computation}
\mathbb E\bigl(\prod_{i=1}^N (y- X_{t,i})\bigr)&=
\sum_{n=0}^N (-1)^n\mathbb E\bigl(e_n^N( X_{t})\bigr)\cdot y^{N-n}
=\sum_{n=0}^N (-1)^n e_n^N(z)\cdot y^{N-n}\\
&=\prod_{i=1}^N (y- z_i)= \frac{1}{l_N^{(\alpha,\beta)}}  \cdot P_N^{(\alpha,\beta)}(y).
\notag\end{align}
This proves the first statement. 
We now use
$l_N^{(\alpha,\beta)}=2^{-N}\binom{2N+\alpha+\beta}{N}$ (see (4.21.6) of \cite{S}) and compare the coefficients in 
 (\ref{det-computation}) and (\ref{jacobi-pol-def}). This in combination with the binomial formulas
easily leads to the second  statement. 
\end{proof}

\begin{remark}
For $p=q$, i.e., $\alpha=\beta$, Eq.~(\ref{det-form-2}) can be written in a simpler way by using the 
$\>_2F_1$-representation (4.7.30) of \cite{S} of the Jacobi polynomials in this case.
In fact, a straightforward computation here implies the following:

If $N=2R$ is even, then 
$ \mathbb E\bigl(e_{n}^N(X_{t})\bigr)= e_{n}^N(z)=0$ for $k$ odd, and, for $n=0,\ldots,R$,
$$\mathbb E\bigl(e_{2n}^N(X_{t})\bigr)= e_{2n}(z)= (-1)^n \frac{R!\cdot n!}{(R-n)!} \cdot
 \frac{(2R+\alpha+1/2 -n)_n}{(1/2+R -n)_n}.$$
Moreover, for  $N=2R+1$ odd, we have $ \mathbb E\bigl(e_{k}^N(X_{t})\bigr)= e_{k}^N(z)=0$ for $k$ even,
 and, for $n=0,\ldots,R-1$,
$$\mathbb E\bigl(e_{2n+1}^N(X_{t})\bigr)= e_{2n+1}^N(z)= (-1)^n \frac{R!\cdot n!}{(R-n)!} \cdot
 \frac{(2R+\alpha+3/2 -n)_n}{(3/2+R -n)_n}.$$
\end{remark}

We now apply Theorem \ref{det-formula} for $t\to\infty$ in order to get a corresponding result for 
$\beta$-Jacobi ensembles:

\begin{corollary}\label{det-formula-uniform} 
Let $k_1,k_2,k_3\in\mathbb R$ with $k_3>0$, $k_2>-1/2$, and  $k_1+k_2>-1/2$.
 Let $X$ be an $A_N$-valued random variable with Lebesgue density
\begin{equation}\label{stationary-dist}
 c_{k_1,k_2,k}
 \prod_{i=1}^N \bigl((1-x_i)^{k_1+k_2-1/2}(1-x_i)^{k_2-1/2}\bigr)
 \prod_{1\le i<j\le N}|x_i-x_j|^{2k_3}.
\end{equation}
Then
\begin{equation}\label{det-form-stat}
\mathbb E\bigl(\prod_{i=1}^N (y- X_{i})\bigr) = \frac{1}{l_N^{(\alpha,\beta)}} \cdot P_N^{(\alpha,\beta)}(y)
  \quad\quad\text{for}\quad y\in \mathbb R
\end{equation}
with the Jacobi polynomial $ P_N^{(\alpha,\beta)}$ with
 leading coefficient $l_N^{(\alpha,\beta)}$ and with
$$\alpha=\frac{1+2k_1+2k_2}{2k_3}-1>-1, \quad \beta=\frac{1+2k_2}{2k_3}-1>-1.$$
Moreover, for $n=0,1,\ldots, N$,
\begin{align}\label{det-form-3}
& \mathbb E\bigl(e_{n}^N(X)\bigr)= \\
&=\frac{2^{N}}{\binom{2N+\alpha+\beta}{N}}
\sum_{l=N-n}^N  (-1)^{N-l} \binom{N}{l} \binom{l}{N-n} 
\frac{(N+\alpha+\beta+1)_l(\alpha+l+1)_{N-l}}{N!2^l}.\notag
\end{align}
\end{corollary}

\begin{proof} By the parameter transform (\ref{parameter-change-k-p}), we have $p,q>N-1$. Moreover, 
(\ref{parameter-change-k-p}) and (\ref{alpha-beta}) lead to the formula for $\alpha,\beta$ in the corollary.
We now consider the associated Jacobi processes $(X_t)_{t\ge0}$ as in  Theorem \ref{det-formula}. 
 It follows from   \cite{RR1} (see in particular Proposition 3.4 there) that the $X_t$ tend for $t\to\infty$
to $X$ in distribution. Therefore, as $A_N$ is compact, the corollary follows from  Theorem \ref{det-formula}. 
\end{proof}

\begin{example} Let $N=1$. Here $k_3$ is irrelevant, and we have from (\ref{jacobi-pol-def}) that
$$ \frac{1}{l_N^{(\alpha,\beta)}} \cdot P_1^{(\alpha,\beta)}(y)= y+\frac{\alpha-\beta}{\alpha+\beta+2}=
y+\frac{k_1}{k_1+2k_2+1}.$$
Eq.~(\ref{det-form-stat}) can be checked here by via classical beta-integrals.
\end{example}

Corollary \ref{det-formula-uniform} is equivalent to Aomoto's Selberg integral \cite{A}
 which involves additional elementary symmetric polynomials in classical Selberg integrals. 
These formulas admit even further generalizations like Kadell's Selberg integral \cite{Ka} where 
Jack polynomials $C_\lambda^{(1/k_3)}$ instead of
 elementary symmetric polynomials are used in  Selberg integrals.

\begin{remark}
If we put $k_1=0$ and use the tranform $x_i\mapsto x_i/\sqrt{k_2}$ ($i=1,\ldots,N$),
then Corollary \ref{det-formula-uniform} leads for $k_2\to\infty$ to 
a corresponding result for $\beta$-Hermite ensembles; see e.g. \cite{FG}.

Moreover, if we use the tranform $x_i\mapsto \frac{\alpha}{2}(x_i+1)$ ($i=1,\ldots,N$)
with $\alpha$ as in  Corollary \ref{det-formula-uniform}, then 
Corollary \ref{det-formula-uniform} leads for $k_1\to\infty$
 to a corresponding result for $\beta$-Laguerre ensembles; see e.g. \cite{FG}.

Corresponding limits are also possible on the level of the diffusions above,
 and one obtains the results in \cite{KVW} for
 multivariate Bessel processes associated with the root systems A and B.
For these  Bessel processes we refer to \cite{CDGRVY} and references there.
\end{remark}

\section{An algebraic explanation of some of the preceding results}

We finally discuss Proposition \ref{elementary-symm-martingale}  from an algebraic point of view.
For this we consider the  Heckman-Opdam Jacobi polynomials $ \widetilde R_{\lambda}$ from 
 Section 2.2 for partitions $\lambda=(\lambda_1,\ldots,\lambda_N)\in\mathbb Z_+$ with
 $\lambda_1\ge\ldots\ge\lambda_N$. It follows from  (\ref{rho-component}) that $ \widetilde R_{\lambda}$ 
is an eigenfunction of $ L_k$ with eigenvalue
\begin{equation}\label{eigen-jacobi} r(\lambda):= -\langle \lambda,\lambda+2\rho(k)\rangle=
-\sum_{i=1}^N \lambda_i( \lambda_i+k_1+2k_2+2k_3(N-i)).
\end{equation}
Moreover, the  $ \widetilde R_{\lambda}$ form an orthogonal basis of $L^2(A_N,w_k)$.

On the other hand,  Proposition \ref{elementary-symm-martingale} and its proof with the comments about 
space-time harmonic functions imply that for $n=1,\ldots,N$ the polynomials
\begin{equation}\label{pol-jac}
q_n(x):= e_n^N(x) + \sum_{l=1}^{n} c_{n,l} \cdot e_{n-l}^N(x)
\end{equation}
with the coefficients  $c_{n,l}$ from (\ref{mart-formel}) are  eigenfunctions of  $\tilde L_k$
and thus of
$ L_k$. In particular, $q_n$ is an eigenfunction of $ L_k$ with eigenvalue
$$-k_3n(p+q-n+1)=-n(1+k_1+2k_2+(2N-n-1)k_3).$$
This is equal to   $r(\lambda(n))$ in (\ref{eigen-jacobi}) for
the partition $\lambda(n):=(1,\ldots,1,0,\ldots,0)$ where 1 appears $n$-times.
Therefore,  the following result is quite natural.

\begin{lemma} Let $k_1,k_2,k_3\in\mathbb R$ with $k_3>0$, $k_2>-1/2$, and  $k_1+k_2>-1/2$. 
Then for each  $n=1,\ldots,N$, the polynomials  $ \widetilde R_{\lambda(n)}$ and $ q_n $ are equal
up to a  multiplicative constant. In particular,
 $ \widetilde R_{\lambda(n)}$ is independent of $k_3$.
\end{lemma}

\begin{proof}
Let  $n=1,\ldots,N$. Then the statement is clear when the eigenvalue  $r(\lambda(n))$
has multiplicity 1. 

Assume now that  $r(\lambda(n))$ has multiplicity $\ge2$, i.e.,
 there exists a partition $\lambda\ne \lambda(n)$ with
\begin{equation}\label{equal-eigen}
\langle \lambda,\lambda+2\rho(k)\rangle=\langle \lambda(n),\lambda(n)+2\rho(k)\rangle.
\end{equation}
As $\rho(k)_1\ge\ldots\ge\rho(k)_N$, a simple monotonicity argument shows that then
 $\lambda$ satisfies $\sum_{i=1}^N \lambda_i\ne n$. On the other hand, if (\ref{equal-eigen}) holds for some
 $\lambda\ne \lambda(n)$ with $\sum_{i=1}^N \lambda_i\ne n$ and some $k_1,k_2,k_3$, then 
 (\ref{equal-eigen}) fails to hold for any slightly modified parameter $k_1$ by (\ref{eigen-jacobi}).
Therefore,  $ \widetilde R_{\lambda(n)}$ and $ q_n $ are equal
up to a  multiplicative constant for these modified $k_1$. An obvious continuity argument now shows that
this equality also holds for the original  $k_1$. This completes the proof.
\end{proof}

\end{document}